\documentclass{article}

\usepackage{fullpage}
\usepackage{verbatim}
\usepackage{graphics}
\usepackage[pdftex]{graphicx}
\usepackage{amsmath, amssymb, amsfonts, amsthm}
\usepackage{url}
\usepackage{setspace}
\usepackage{algorithmic}
\newtheorem{theorem}{Theorem}
\newtheorem{conjecture}[theorem]{Conjecture}
\newtheorem{lemma}[theorem]{Lemma}

\newtheorem{corollary}[theorem]{Corollary}
\newtheorem{proposition}[theorem]{Proposition}

\begin{document}

\author{Ben D. Lund\footnote{lund.ben@gmail.com. Department Of Computer Science, 814 Rhodes Hall, Cincinnati, OH 45221}\\
George B. Purdy\footnote{george.purdy@uc.edu. Department Of Computer Science, 814 Rhodes Hall, Cincinnati, OH 45221}\\
Justin W. Smith\footnote{smith5jw@mail.uc.edu. Department Of Computer Science, 814 Rhodes Hall, Cincinnati, OH 45221} }

\title{A Bichromatic Incidence Bound and an Application}
\maketitle

\begin{abstract}

We prove a new, tight upper bound on the number of incidences between points and hyperplanes in Euclidean d-space.
Given $n$ points, of which $k$ are colored red, there are $O_d(m^{2/3}k^{2/3}n^{(d-2)/3} + kn^{d-2} + m)$ incidences between the $k$ red points and $m$ hyperplanes spanned by all $n$ points provided that $m = \Omega(n^{d-2})$.
For the monochromatic case $k = n$, this was proved by Agarwal and Aronov \cite{AA92}.

We use this incidence bound to prove that a set of $n$ points, no more than $n-k$ of which lie on any plane or two lines, spans $\Omega(nk^2)$ planes.
We also provide an infinite family of counterexamples to a conjecture of Purdy's \cite{EP95} on the number of hyperplanes spanned by a set of points in dimensions higher than 3, and present new conjectures not subject to the counterexample.
\end{abstract}

\section{Introduction}
We consider higher dimensional generalizations of two classical extremal problems from planar combinatorial geometry, and use our upper bound on the number of incidences between points and hyperplanes to obtain a lower bound on the number of planes spanned by a point set.

A point and line are incident if the point is on the line.
Szemer\'edi and Trotter \cite{SzTr83} proved the tight result that there are $O(m^{2/3}n^{2/3} + m + n)$ incidences between $n$ points and $m$ lines in the plane (Lemma \ref{SzemerediTrotter}, below).
Since then, various generalizations of this upper bound have been proved or conjectured - see \cite{PaSh04} for a survey.

One natural generalization is to determine an upper bound on the number of incidences between points and hyperplanes in $\mathbb{E}^d$.
In order to prove an interesting upper bound for this problem, we must restrict the class of admissible arrangements.
Without such a restriction, the trivial upper bound of $mn$ can be attained by placing the $n$ points on a single $(d-2)$-flat that is covered by each of the $m$ hyperplanes.
A number of upper bounds with various restrictions on admissible arrangements have been proved  \cite{EdGuSh90, EdSh90, AA92, BrKn03, ElTo05}.

A hyperplane in $\mathbb{E}^d$ is spanned by a set of points if it passes through $d$ points that do not all lie on a $(d-2)$-flat.
Agarwal and Aronov \cite{AA92} showed that there are $O_d(m^{2/3}n^{d/3} + n^{d-1})$ incidences between $n$ points and $m$ hyperplanes spanned by the points, provided that $m = \Omega(n^{d-2})$.
(The subscript on $O_d$ indicates that the constant implicit in the asymptotic notation depends on $d$).
This problem had been previously considered by Edelsbrunner et. al. \cite{EdGuSh90}, and a construction achieving the asymptotic bound was provided by Edelsbrunner \cite[p. 112]{Ed87}.

In section \ref{sec:BichromaticIncidenceBound}, we show that there are $O_d(m^{2/3}k^{2/3}n^{(d-2)/3} + kn^{d-2} +m)$ incidences between $k$ red points and $m$ hyperplanes if the hyperplanes are spanned jointly by the red points and $n-k$ blue points, provided that $m = \Omega(n^{d-2})$.
In the same section, we provide a construction achieving this asymptotic bound.
This bound is equivalent to Agarawal and Aronov's in the limiting case that all of the hyperplanes are red (i.e., $k = n$).

Although the bichromatic incidence bound is interesting in its own right, our investigation of it is motivated by our interest in a second question:
What is the minimum number of hyperplanes spanned by a set of points?
To obtain a non-trivial answer to this question, some restriction on the point set must be assumed even in two dimensions, since a set of collinear points determines exactly one line.

The specific results we are most interested in extending are the weak Dirac \cite{SzTr83, Be83} and Beck-Erd\H{o}s \cite{Be83} theorems (Lemmas \ref{WeakDirac} and \ref{BeckErdos}, below).

In contrast to the numerous higher dimensional generalizations of the Szemer\'edi-Trotter theorem, there are few extremal results on the number of hyperplanes spanned by a set of points in more than two dimensions.
Hansen proved that a set of points in $\mathbb{E}^d$ that do not all lie on a single hyperplane determines at least one ordinary hyperplane, where an ordinary hyperplane is defined as a hyperplane with all but one of its points on a $(d-2)$-flat \cite{Han65}.
Beck proved that a set of points in $\mathbb{E}^d$, not too many of which lie on any single hyperplane, determines $\Omega_d(n^d)$ hyperplanes \cite{Be83} (Lemma \ref{TwoExtremes}).

Our main result on the number of hyperplanes spanned by a set of points, presented in section \ref{sec:Beck-ErdosR3}, is a three dimensional analog to the Beck-Erd\H{o}s theorem.
We show that a set of points in $\mathbb{E}^3$, no more than $n-k$ of which lie on a plane or any pair of skew lines, determines $\Omega(nk^2)$ planes.

In section \ref{sec:WeakPurdy}, we provide an infinite family of counterexamples to a conjecture of Purdy on the number of hyperplanes spanned by a set of points in $\mathbb{E}^d$ \cite{EP95}, and provide a modified conjecture not refuted by the counterexample.
We also conjecture a generalization of the Beck-Erd\H{o}s theorem that encompasses the three dimensional Beck-Erd\H{o}s analog we prove in section \ref{sec:Beck-ErdosR3}.

\subsection{Theorems in planar combinatorial geometry}
\label{sec:BasicTheorems}

Our proofs in higher dimensions depend on results from planar combinatorial geometry.
Although these theorems were proved in the plane, they apply to points and lines in any dimension since collinear points remain collinear and distinct under projection to a suitable plane.

We use the Szemer\'edi-Trotter incidence bound.
\begin{lemma}[Szemer\'edi, Trotter \cite{SzTr83}]\label{SzemerediTrotter}
There are $O(m^{2/3}n^{2/3} + m + n)$ incidences between $n$ points and $m$ lines in $\mathbb{E}^2$.
\end{lemma}

We also rely on three lower bounds on the number of lines spanned by a set of points in the plane.
The first of these, conjectured independently by Dirac \cite{Dir51} and Motzkin \cite{Mot51}, and proved independently by Szemer\'edi and Trotter \cite{SzTr83} and Beck \cite{Be83}, is known as the weak Dirac theorem.
\begin{lemma}[Szemer\'edi, Trotter, Beck \cite{SzTr83, Be83}]\label{WeakDirac}
A set of $n$ points in $\mathbb{E}^2$ that are not all collinear contains a point incident to $\Omega(n)$ lines.
\end{lemma}

The second result we use on the number of lines spanned by a set of points was conjectured by Erd\H{o}s \cite{Er75, Er81} and proved by Beck. Purdy had previously shown it to be a consequence of the weak Dirac \cite{Pur81}. It is known as the Beck-Erd\H{o}s theorem. 
\begin{lemma}[Beck \cite{Be83}]\label{BeckErdos}
A set of $n$ points in $\mathbb{E}^2$ of which at most $n-k$ are collinear determines $\Omega(nk)$ lines.
\end{lemma}

We will refer to the third result on lines spanned by points that we use as Beck's lemma.
It was first stated and proved by Beck, and was used by Beck to prove the Beck-Erd\H{o}s and weak Dirac theorems.
\begin{lemma}[Beck \cite{Be83}]\label{TwoExtremesInE2}
A set of $n$ points in $\mathbb{E}^2$ of which at most $n/100$ are collinear determines $\Omega(n^2)$ lines.
\end{lemma}

\section{Bichromatic incidence bound}
\label{sec:BichromaticIncidenceBound}

\newcommand{\diracC}{\ensuremath{c_{dirac}}}
\newcommand{\hyps}{\ensuremath{\mathcal{H}}}
\newcommand{\hypnum}[1]{\ensuremath{H_{#1}}}
\newcommand{\hyp}[2]{\ensuremath{H^{(#2)}_{#1}}}
\newcommand{\linep}[2]{\ensuremath{L_{#1}^{(#2)}}}
\newcommand{\lines}[1]{\ensuremath{\mathcal{L}}^{(#1)}}
\newcommand{\redinc}[1]{\ensuremath{\mathcal{I}_{R}^{#1}}}
\newcommand{\blueinc}[1]{\ensuremath{\mathcal{I}_{B}^{#1}}}
\newcommand{\redlines}[1]{\ensuremath{\mathcal{L}_{R}^{(#1)}}}
\newcommand{\bluelines}[1]{\ensuremath{\mathcal{L}_{B}^{(#1)}}}

\newcommand{\khypsind}[1]{\ensuremath{k'}}
\newcommand{\nhypsind}[1]{\ensuremath{n'}}
\newcommand{\redhypsind}[1]{\ensuremath{\mathcal{H'}^{(#1)}_{R}}}
\newcommand{\bluehypsind}[1]{\ensuremath{\mathcal{H'}^{(#1)}_{B}}}

\newcommand{\redhyps}{\ensuremath{\mathcal{H}_{R}}}
\newcommand{\redhypsinc}[1]{\ensuremath{\mathcal{H}^{(#1)}_{R}}}
\newcommand{\bluehyps}{\ensuremath{\mathcal{H}_{B}}}
\newcommand{\bluehypsinc}[1]{\ensuremath{\mathcal{I}^{(#1)}_{B}}}
\newcommand{\hypspoint}[1]{\ensuremath{\mathcal{H}^{(#1)}}}
\newcommand{\incds}[2]{\ensuremath{\mathbf{I^d}\left(#1,#2\right)}}
\newcommand{\verts}{\ensuremath{\mathcal{P}}}
\newcommand{\vassigned}[1]{\ensuremath{\mathcal{P}^{(#1)}}}

In this section, we investigate upper bounds on the number of incidences between the vertices of an arrangement of red and blue hyperplanes, and the red hyperplanes of the arrangement.
This is dual to the incidence problem stated in the introduction.

\subsection{Incidences between hyperplanes and vertices of their arrangement}

Our results generalize those of Agarwal and Aronov \cite{AA92}, and Edelsbrunner \cite{Ed87}.
The problem they considered is as follows.

In $\mathbb{E}^d$, given a set $\hyps$ of hyperplanes and a subset $\verts$ of the vertices of their arrangement, let $\mathbf{I^d}(\verts, \hyps)$ denote the number of incidences between $\hyps$ and $\verts$.
Let $\mathbf{I^d}(m,n)$ be the maximum number of incidences over all such sets of $n$ hyperplanes and $m$ vertices of their arrangement.
That is, let
\[\mathbf{I^d}(m, n) = \max_{\substack{|\verts|=m\\
|\hyps|=n}}\mathbf{I^d}(\verts,\hyps).\]

Since each point-hyperplane pair determines at most $1$ incidence, $\mathbf{I^d}(m,n) = O(mn)$.
Edelsbrunner proved that this is tight in the case that $m=O(n^{d-2})$.
\begin{lemma}[Edelsbrunner, \cite{Ed87}]\label{mnBound}
If $d \geq 2$ and $m=O(n^{d-2})$, then
\[\mathbf{I^d}(m,n) = \Theta_d(mn).\]
\end{lemma}

$\mathbf{I^d}(m,n)$ increases monotonically with $m$, so Lemma \ref{mnBound} implies that $\mathbf{I^d}(m,n) = \Omega_d(n^{d-1})$ when $m=\Omega(n^{d-2})$.
Edelsbrunner \cite{Ed87} also showed that $\mathbf{I^d}(m,n) = \Omega_d(m^{2/3}n^{d/3})$ when $m=\Omega(n^{d-2})$.
Agarwal and Aronov proved the corresponding upper bound for these cases \cite{AA92}.
\begin{lemma}[Agarwal, Aronov]\label{AgarwalAronov}
If $d \geq 2$ and $m = \Omega(n^{d-2})$, then
\[\mathbf{I^d}(m,n) = \Theta_d(m^{2/3}n^{d/3} + n^{d-1}).\]
\end{lemma}

\subsection{A bichromatic generalization}

In $\mathbb{E}^d$, given a set $\redhyps$ of red hyperplanes, a set $\bluehyps$ of blue hyperplanes, and a subset $\verts$ of the vertices of their combined arrangement, let $\mathbf{I^d}(\verts, \redhyps, \bluehyps)$ denote the number of incidences between $\redhyps$ and $\verts$.
Let $\mathbf{I^d}(m,k,n)$ be the maximum number of red hyperplane-point incidences over all such sets of $k$ red hyperplanes, $n-k$ blue hyperplanes, and $m$ vertices of their combined arrangement.
That is, let
\[\mathbf{I^d}(m, k, n) = \max_{\substack{|\verts|=m\\
|\redhyps|=k\\
|\bluehyps|=n-k}}\mathbf{I^d}(\verts,\redhyps,\bluehyps).\]

If $k = n$, there are no blue hyperplanes and this is exactly the problem considered by Edelsbrunner, Agarwal and Aronov.

\begin{theorem}\label{mnBichromaticBound}
If $d \geq 2$ and $m = O(n^{d-2})$, then
\[\mathbf{I^d}(m,k,n) = \Theta_d(mk).\]
\end{theorem}
\begin{proof}
The upper bound is trivial.

For the lower bound, take an arrangement of $n$ hyperplanes and $m$ vertices of their arrangement such that the total number of incidences between the hyperplanes and vertices is $\Omega(mn)$.
Lemma \ref{mnBound} implies that such an arrangement exists.
Choose $k$ hyperplanes of the arrangement such that no other set of $k$ hyperplanes is incident to more of the vertices and color them red.
By construction, the average number of incidences between the red hyperplanes and vertices is at least as large as the average number of incidences between all the hyperplanes and the vertices.
Since the average number of incidences between all the hyperplanes and vertices is $\Omega(m)$, the total number of incidences between red hyperplanes and vertices must be $\Omega(mk)$.

\end{proof}

\begin{theorem}[Bichromatic incidence bound]\label{BichromaticAgarwalAronov}
If $d \geq 2$ and $m = \Omega(n^{d-2})$, then
\[\mathbf{I^d}(m, k, n) = \Theta_d(m^{2/3}k^{2/3}n^{(d-2)/3} + kn^{d-2} + m).\]
\end{theorem}

If $k=n$ (i.e., there are no blue hyperplanes), then this simplifies to $\Theta_d(m^{2/3}n^{d/3} + n^{d-1} + m)$.
Since the points are vertices of the arrangement, $m = O(n^d)$, so $m = O(m^{2/3}n^{d/3})$.
Substituting, we obtain $\Theta_d(m^{2/3}n^{d/3} + n^{d-1})$, the bound proved by Agarwal and Aronov (Lemma \ref{AgarwalAronov}).

We prove Theorem \ref{BichromaticAgarwalAronov} in two parts.
First, following Agarwal and Aronov \cite{AA92}, we show an upper bound on $\mathbf{I^d}(m,n,k)$, and then, to show this is the best possible, we modify Edelsbrunner's \cite[p113]{Ed87} construction that achieves the upper bound.

\subsection{Proof of the upper bound for Theorem \ref{BichromaticAgarwalAronov}}

\begin{lemma}\label{bichromaticUpperBound}If $d \geq 2$ and $m = \Omega(n^{d-2})$, then
\[\mathbf{I^d}(m,k,n) = O_d(m^{2/3}k^{2/3}n^{(d-2)/3} + kn^{d-2} + m).\]
\end{lemma}

\begin{proof}
In $\mathbb{E}^d$, let $\redhyps$ be a set of $k$ red hyperplanes, and let $\bluehyps$ be a set of $n-k$ blue hyperplanes such that $\redhyps \cap \bluehyps = \emptyset$.
Let $\hyps$ be their combined arrangement.
Let $\verts$ be a subset of the vertices of $\hyps$, each incident to at least one red hyperplane, with $ \left|\verts\right| = m$.

Lemma \ref{SzemerediTrotter} (the Szemeredi-Trotter incidence bound)
does not require the points to be vertices of the arrangement of lines,
so the theorem is true for $\mathbb{E}^2$.
This will be the base case for induction.

Let $P$ be an arbitrary point of $\verts$, and let 
$\hypspoint{P} = \{\hyp{1}{P},\hyp{2}{P},\ldots,\hyp{t}{P}\} \subseteq \hyps$ 
be the set of hyperplanes that contain $P$. 
Let $\redhypsinc{P}$ be the set of red hyperplanes that contain $P$.

Place a two dimensional plane $M$ in general position, so that the hyperplanes of $\hypspoint{P}$ intersect $M$ at lines 
$\lines{P} = \{\linep{1}{P},\linep{2}{P},\ldots,\linep{t}{P}\}$.
We can place $M$ so that no two members of $\lines{P}$ are parallel.
Let each $\linep{i}{P}$ correspond to hyperplane $\hyp{i}{P}$ and have the same color as $\hyp{i}{P}$.
Let $\redlines{P}\subseteq\lines{P}$ be the red lines, and let $\bluelines{P}\subseteq\lines{P}$ be the blue lines.

Since $P$ is a vertex spanned by the hyperplanes of $\hypspoint{P}$, 
the lines of $\lines{P}$ are not all concurrent; however, the red lines of $\redlines{P}$ might be.
If the red lines are concurrent, then there must exist a blue line $\linep{i}{P} \in \bluelines{P}$ that intersects all of the red lines at distinct points.
Otherwise, by the dual weak Dirac there exists a red line $\linep{i}{P} \in \redlines{P}$ and a constant \diracC, 
such that $\linep{i}{P}$ intersects at least $\diracC|\redlines{P}|$ other red lines at distinct points.
In either case, there exists a line $\linep{i}{P}$ on $M$ that intersects the lines of $\redlines{P}$ in at least $\diracC|\redlines{P}|$ distinct points,
and so the hyperplane $\hyp{i}{P}$ corresponding to $\linep{i}{P}$ intersects the hyperplanes of $\redhypsinc{P}$ in at least $\diracC|\redlines{P}|=\diracC|\redhypsinc{P}|$ distinct $(d-2)$-flats.
Since $P$ is spanned by $\hyps$, it is also spanned by the intersection of $(d-2)$-flats in $\hyp{i}{P}$.

Assign $P$ to $\hyp{i}{P}$, i.e., let $f(P)=\hyp{i}{P}$.
The process of assignment may be repeated for each point in $\verts$, assigning them to hyperplanes in $\hyps$.

Select an arbitrary hyperplane $\hypnum{i}$ of $\hyps$.
Let $\vassigned{i}$ be the set vertices assigned to hyperplane $\hypnum{i}$. 
That is, $\vassigned{i} = \{P \in \verts: f(P) = \hypnum{i}\}$.
Let $m_i = |\vassigned{i}|$.
Let $\redhypsind{i}$ be the set of $(d-2)$-flats formed by the intersection of the hyperplanes of $\redhyps$ with $\hypnum{i}$, and let $\bluehypsind{i}$ be the $(d-2)$-flats formed by the intersection of the hyperplanes of $\bluehyps$ with $\hypnum{i}$.
Let $\khypsind{i} = |\redhypsind{i}|$, and let $\nhypsind{i} = |\redhypsind{i} \cup \bluehypsind{i}|$.

We now have an arrangement, $\redhypsind{i} \cup \bluehypsind{i}$, of $(d-2)$-flats contained in $(d-1)$-dimensional space, $\hypnum{i}$,
that determine a set of points, $\vassigned{i}$.
Note that for a particular $P \in \vassigned{i}$, the number of incidences between $P$ and the $(d-2)$-flats of $\redhypsind{i}$
is at least $\diracC$ times the number of incidences it has with the hyperplanes $\redhyps$.
We apply the inductive hypothesis to determine an upper bound on \incds{\vassigned{i}}{\redhyps},
\begin{equation}
\label{eqn:baa-ind}
\begin{split}
\incds{\vassigned{i}}{\redhyps}& \leqslant
\frac{1}{\diracC}\incds{\vassigned{i}}{\redhypsind{i}}\\
& = \frac{1}{\diracC} O_{d-1}\left(m_i^{2/3}\khypsind{i}^{2/3}\nhypsind{i}^{(d-3)/3} + \khypsind{i}\nhypsind{i}^{d-3} + m_i\right)\\
& = O_d\left(m_i^{2/3}\khypsind{i}^{2/3}\nhypsind{i}^{(d-3)/3} + \khypsind{i}\nhypsind{i}^{d-3} + m_i\right)\\
& = O_d\left(m_i^{2/3}k^{2/3}n^{(d-3)/3} + kn^{d-3} + m_i\right).
\end{split}
\end{equation}

The same logic may be applied to the remaining $n-1$ hyperplanes of $\hyps$.
Since each point $P \in \verts$ is assigned to exactly one \vassigned{i},
\[
\incds{\verts}{\redhyps} = \sum_{i=1}^n\incds{\vassigned{i}}{\redhyps}.
\]

From the upper bound on  $\incds{\vassigned{i}}{\redhyps}$ (inequality \ref{eqn:baa-ind}), it follows that 
\[
\incds{\verts}{\redhyps} = O_d\left(\sum_{i=0}^n\left(m_i^{2/3}k^{2/3}n^{(d-3)/3} + kn^{d-3}+m_i\right)\right).
\]

Since each point $P \in \verts$ is in exactly one of the sets $\vassigned{i}$,
we see that $m = \sum_{i=1}^n m_i$.
Since $\varphi(x)=x^{2/3}$ is concave, Jensen's inequality implies
 $\sum_{i=1}^n m_i^{2/3} \leq n^{1/3}(\sum_{i=1}^n m_i)^{2/3}$,
and thus,
\[
\mathbf{I^d}(m,k,n) = O_d\left(m^{2/3}k^{2/3}n^{(d-2)/3} + kn^{d-2} + m\right).
\]
\end{proof}

\subsection{Proof of the lower bound for Theorem \ref{BichromaticAgarwalAronov}}
The upper bound has three terms.
If $m = \Omega(k^2n^{d-2})$, then the $O_d(m)$ term dominates.
If $m = O(k^{1/2}n^{d-2})$, then the $O_d(kn^{d-2})$ term dominates.
Otherwise, the $O_d(m^{2/3}k^{2/3}n^{(d-2)/3})$ term dominates.

We may very easily achieve $1$ incidence for each point, so
\[\mathbf{I^d}(m,k,n) = \Omega(m).\]
Since $\mathbf{I^d}(m,k,n)$ increases monotonically with $m$, Theorem \ref{mnBichromaticBound} immediately implies that if $m = \Omega(n^{d-2})$, then
\[\mathbf{I^d}(m,k,n) = \Omega_d(kn^{d-2}).\]

We will modify a construction presented by Edelsbrunner \cite[p. 112]{Ed87} to demonstrate the third term in the lower bound.
We need the following lower bound  in two dimensions, customarily attributed to Paul Erd\H{o}s: 
\begin{lemma}[Edelsbrunner, \cite{Ed87}]\label{E2Construction}
\[\mathbf{I^2}(m,n) = \Omega(m^{2/3}n^{2/3}).\]
\end{lemma}

\begin{lemma}
\[\mathbf{I^d}(m,k,n) = \Omega_d(m^{2/3}k^{2/3}n^{(d-2)/3}).\]
\end{lemma}
\begin{proof}
If $m = \Omega(k^2n^{d-2})$, this follows immediately from the fact that $\mathbf{I^d}(m,k,n) = \Omega(m)$, so we will assume that $m = O(k^2n^{d-2})$.
If $m = O(n^{d-2})$, this follows immediately from Theorem \ref{mnBichromaticBound}, so we will assume that $m = \Omega(n^{d-2})$.
If $k = \Omega(n)$, the conclusion follows immediately from the lower bound established by Edelsbrunner, so we will assume that $n-k = \Omega(n)$.

We will construct a set of hyperplanes $\mathcal{H}$ from the disjoint union of sets $\redhyps, \mathcal{H}_1, ..., \mathcal{H}_{d-2}$, to be defined later.
The set of vertices of $\mathcal{H}$ is $\mathcal{P}$, and we will count the incidences between $\mathcal{P}$ and $\redhyps$.

We will use a coordinate system $(x_1, x_2, ..., x_d)$.

Define
\[p = c_0\left\lfloor m/n^{d-2} \right\rfloor,\]
with $c_0$ to be set later.
Since $m = O(k^2n^{d-2})$, we know that 
$p \leq c_1k^2$, for some $c_1$ depending on $c_0$.

The set $\redhyps$ contains $k$ red hyperplanes normal to a plane $\pi$, with $\pi$ defined as
\[\pi: x_1 = x_2 = ... = x_{d-2}=0.\]
The intersection of the arrangement $\mathcal{H}_R$ with $\pi$ is a set of lines $\mathcal{L}_R$ intersecting in a set of points $\mathcal{P}_R$ with $|\mathcal{P}_R| = p$ and $\mathbf{I^2}(\mathcal{L}_R, \mathcal{P}_R) = \Omega(p^{2/3}k^{2/3})$.
Since $p \leq c_1k^2$, Lemma \ref{E2Construction} implies that there exists a constant $c_0$ such that this arrangement is guaranteed to exist; choose $c_0$ accordingly.
Since the hyperplanes of $\mathcal{H}_R$ are normal to $\pi$, the intersection of $\mathcal{H}_R$ with any plane parallel to $\pi$ will be combinatorially equivalent to the arrangement of $\mathcal{L}_R$.

The sets $\mathcal{H}_1, ..., \mathcal{H}_{d-2}$ each contain $\lfloor (n-k)/(d-2) \rfloor$ parallel blue hyperplanes.
Place the sets of hyperplanes so that any $d-2$ hyperplanes, one from each set $\mathcal{H}_1, ..., \mathcal{H}_{d-2}$, intersect in a unique common plane parallel to $\pi$.
Formally, we can define
\[\mathcal{H}_i = \{x_i = j | j=0,1,...,\lfloor (n-k)/(d-2) \rfloor - 1\}\]
for $1 \leq i \leq d-2$.
Since there are $\lfloor (n-k)/(d-2) \rfloor^{d-2}$ ways to choose one hyperplane from each set $\mathcal{H}_1,...,\mathcal{H}_{d-2}$, and each plane forming their intersection accounts for $\Omega(p^{2/3}k^{2/3})$ intersections between vertices of the overall arrangement and the red hyperplanes of $\mathcal{H}_0$, we can conclude that
\[\mathbf{I^d}(\mathcal{P}, \mathcal{H}_R, \mathcal{H}_1 \cup \mathcal{H}_2 \cup ... \cup \mathcal{H}_{d-2}) = \Omega(p^{2/3}k^{2/3}(n-k)^{d-2}) = \Omega(p^{2/3}k^{2/3}n^{d-2}).\]
Here, $|\mathcal{P}| = p\lfloor(n-k)/(d-2)\rfloor^{d-2} = \Theta(m)$.
Since $p^{2/3}k^{2/3}n^{d-2} = \Theta(m^{2/3}k^{2/3}n^{(d-2)/3})$, this proves the theorem.
\end{proof}

Together, the above upper and lower bounds imply Theorem \ref{BichromaticAgarwalAronov}.

\section{Three dimensional analog of the Beck-Erd\H{o}s theorem}
\label{sec:Beck-ErdosR3}

In this section, we prove the following three dimensional analog of the Beck-Erd\H{o}s theorem for points and lines.

\begin{theorem}\label{BeckErdos3}
A set of $n$ points in $\mathbb{E}^3$, of which no more than $n-k$ are on any plane or on any pair of skew lines, spans $\Omega (nk^2)$ planes.
\end{theorem}

\subsection{Planes from incidences}
Our basic strategy is to find $\Omega(nk^2)$ point-hyperplane incidences, and then to apply the following corollary of Theorem \ref{BichromaticAgarwalAronov} to show that the lower bound on point-hyperplane incidences implies a lower bound on hyperplanes.
\begin{lemma}\label{bichromaticConsequence}
Let $\mathcal{P}$ be a set of $n$ points in $\mathbb{E}^d$, and let $\mathcal{K} \subseteq \mathcal{P}$ with $|\mathcal{K}|=k$.
If there are $\Omega(n^{d-2}k^2)$ incidences between points of $\mathcal{K}$ and hyperplanes spanned by $\mathcal{P}$, then there are $\Omega(n^{d-2}k^2)$ hyperplanes spanned by $\mathcal{P}$.
\end{lemma}
\begin{proof}
By the dual of Theorem \ref{BichromaticAgarwalAronov}, there are $O(m^{2/3}n^{(d-2)/3}k^{2/3} + n^{d-2}k + m)$ incidences between $k$ points of $\mathcal{P}$ and the $m$ hyperplanes spanned by $\mathcal{P}$.
There are three cases to consider.

\it Case 1. \rm
\[n^{d-2}k^2 = O(n^{d-2}k).\]
This implies that $k = O(1)$.
Since $O(1)$ points have a total of $\Omega(n^{d-2})$ incidences, some of them must have $\Omega(n^{d-2})$ incidences, so there must be $\Omega(n^{d-2})$ hyperplanes.

\it Case 2. \rm
\[n^{d-2}k^2 = O(m).\]
It immediately follows that $m = \Omega(n^{d-2}k^2)$.

\it Case 3. \rm
\[n^{d-2}k^2 = O((m^{2/3}n^{(d-2)/3}k^{2/3}).\]
Solving for m gives $m=\Omega(n^{d-2}k^2)$.
\end{proof}

\subsection{Combinatorial preliminaries}
It will be helpful to establish two essentially combinatorial results (Corollary \ref{pigeonholeFlats} and Lemma \ref{longLineLemma}) for application in Case 2 of the proof of Theorem \ref{BeckErdos3} (below).
\begin{lemma}[Modified pigeonhole principle]\label{pigeonhole}
Let $0 < c \leq 1$. If $\lceil ck^a \rceil$ discrete objects are allocated to $k$ containers, none of which can contain more than $k^{a-1}$ objects, then at least $\lfloor ck^{a-1}/2 \rfloor$ of the objects must be in each of at least $ck/2$ containers.
\end{lemma}
\begin{proof}
Partition the $k$ containers into two sets:
set $\mathcal{A}$ has the containers with at least $\lfloor ck^{a-1}/2 \rfloor$ objects; set $\mathcal{B}$ has the containers with fewer than $\lfloor ck^{a-1}/2 \rfloor$ objects.
Let $n_a$ be the number of objects in all the containers of $\mathcal{A}$, and let $n_b$ be the number of objects in all the containers of $\mathcal{B}$.

Since no container has more than $k^{a-1}$ objects,
\[n_a \leq |\mathcal{A}|k^{a-1}.\]
Since every member of $\mathcal{B}$ has fewer than $\lfloor ck^{a-1}/2 \rfloor$ objects,
\[\lfloor ck^{a-1}/2 \rfloor|\mathcal{B}| > n_b.\]
Clearly, $n_b = \lceil ck^a \rceil - n_a$ and $|\mathcal{B}| = k - |\mathcal{A}|$, so
\begin{eqnarray*}
\lfloor ck^{a-1}/2 \rfloor(k - |\mathcal{A}|) & > & \lceil ck^a \rceil - |\mathcal{A}|k^{a-1} \\
|\mathcal{A}|k^{a-1} > |\mathcal{A}|(k^{a-1} - \lfloor ck^{a-1}/2 \rfloor) & > & \lceil ck^a \rceil - k\lfloor ck^{a-1}/2 \rfloor \geq ck^a/2.
\end{eqnarray*}
Therefore,
\[ |\mathcal{A}|  >  ck/2.\]
\end{proof}

When we apply this lemma below, the ``objects" in the above lemma are incidences and the ``containers" are points.

\begin{corollary}\label{pigeonholeFlats}
Let $0 < c\leq 1$. If there are at least $ck^a$ incidences between a set of $d$-flats and a set of points, and if no point is incident to more than $k^{a-1}$ $d$-flats, then at least $ck/2$ points must each be incident to at least $\lfloor ck^{a-1}/2 \rfloor$ $d$-flats.
\end{corollary}

This is an immediate consequence of Lemma \ref{pigeonhole}.

\begin{lemma}\label{longLineLemma}
Let $\mathcal{T}$ be a set of $n$ points in $\mathbb{E}^3$ such that no more than $n'$ lie on any plane or pair of lines.
If there is a subset $\mathcal{L} \subset \mathcal{T}$ of at least $n'-k$ collinear points, then
\begin{enumerate}
\item no subset of $\mathcal{T} \setminus \mathcal{L}$ with more than $k$ points is collinear, and
\item no subset of $\mathcal{T} \setminus \mathcal{L}$ with more than $k$ points is on a plane with the line covering $\mathcal{L}$.
\end{enumerate}
\end{lemma}

\begin{proof}
Let $\mathcal{V} \subseteq \mathcal{T} \setminus \mathcal{L}$ with more than $k$ points that is collinear or on a plane with $\mathcal{L}$.
The union of $\mathcal{L}$ and $\mathcal{V}$  will form a set of more than $n'$ points.
\end{proof}

\subsection{Proof of Theorem \ref{BeckErdos3}}

\begin{proof}
Let the full set of $n$ points be $\mathcal{T}$.
A plane or pair of lines that is incident to at least as many points as any other plane or pair of lines exists; we will denote the set of $n-x$ points incident to this plane or pair of lines by $\mathcal{H}$, and denote $\mathcal{T} \setminus \mathcal{H}$ by $\mathcal{X}$.
In other words:
\begin{eqnarray*}
\mathcal{T} & = & \mathcal{H} \cup \mathcal{X} \\
|\mathcal{T}| & = & n \\
|\mathcal{X}| & = & x \geq k \\
|\mathcal{H}| & = & n-x \leq n-k.
\end{eqnarray*}
Since $k \leq x$, it will be sufficient to show that $\Omega(nx^2)$ planes are spanned by $\mathcal{T}$. 

\it Case 1. \rm 
Suppose that no plane contains more than $\beta_3n$ points of $\mathcal{T}$, where $\beta_3$ is a constant from Lemma \ref{TwoExtremes}.
The following ``two extremes" lemma covers this case.
\begin{lemma}[Beck \cite{Be83}]\label{TwoExtremes}
There are positive constants $\beta_d$ and $\gamma_d$ depending only on the dimension such that any set of $n$ points in $\mathbb{E}^d$ of which fewer than $\beta_d n$ points are on any single hyperplane spans at least $\gamma_dn^d$ hyperplanes, for $d \geq 2$.
\end{lemma}

If $|\mathcal{H}| \leq \beta_3n$, then no plane contains more than $\beta_3n$ points, which is case 1.
For the remaining three cases, we may assume that $|\mathcal{H}| > \beta_3n$.

\it Case 2. \rm
Let $\mathcal{L}$ be the largest set of collinear points in $\mathcal{H}$, and suppose $|\mathcal{L}| = |\mathcal{H}|-x'$ with $x' \leq \min(c_ex/2, \beta_3x/2)$, 
where $c_e$ is the constant from Lemma \ref{BeckErdos}.

Let $\Lambda_r$ be the set of lines spanned by $\mathcal{X}$ that do not intersect the line that covers $\mathcal{L}$.

\begin{proposition}\label{RExists}
There is a set of points $\mathcal{R} \subset \mathcal{X}$ with $|\mathcal{R}| = \Omega(x)$ such that each point of $\mathcal{R}$ is incident to $\Omega(x)$ lines of $\Lambda_r$.
\end{proposition}
\begin{proof}
Let $l_t$ be the number of lines spanned by $\mathcal{X}$, let $l_s$ be the number of lines spanned by $\mathcal{X}$ that intersect the line covering $\mathcal{L}$, and let $l_r = |\Lambda_r|$.
In other words,
\[l_r = l_t - l_s.\]

Since $|\mathcal{L}| = |\mathcal{H}| - x'$, Lemma \ref{longLineLemma} implies that no subset of $\mathcal{X}$ larger than $x'$ is collinear.
Since $x' < x/2$, Lemma \ref{BeckErdos} implies that
\[l_t > c_ex^2/2. \]

Lemma \ref{longLineLemma} also implies that no subset of $\mathcal{X}$ larger than $x'$ is on a plane that covers $\mathcal{L}$.
Consequently, no point of $\mathcal{X}$ is on more than $x'-1$ lines spanned by $\mathcal{X}$ that intersect $\mathcal{L}$.
Summing over all $x$ points of $\mathcal{X}$ and dividing by $2$ since each line spanned by $\mathcal{X}$ is incident to at least two points,
\[l_s \leq xx'/2 \leq  c_ex^2/4\]

Since $l_r = l_t - l_s$,
\[l_r > c_ex^2/2 - c_ex^2/4 \geq c_ex^2/4 \geq c_ax^2\]
for some positive constant $c_a$.

Since $l_r = \Omega(x^2)$ and each line of $\Lambda_r$ is on at least two points of $\mathcal{X}$, there must be $\Omega(x^2)$ incidences between $\Lambda_r$ and points of $\mathcal{X}$.
From this, Corollary \ref{pigeonholeFlats} implies that there exists a set of $\Omega(x)$ points each incident to $\Omega(x)$ lines of $\Lambda_r$, completing the proof of Proposition \ref{RExists}.
\end{proof}

Let $P$ be an arbitrary point in $\mathcal{R}$, and let $\Pi_\mathcal{T}$ be the set of planes spanned by $\mathcal{T}$.
Let $\mathbf{I}(P, \Pi_\mathcal{T})$ be the number of incidences between $P$ and $\Pi_\mathcal{T}$.

\begin{proposition}\label{IP}
\[\mathbf{I}(P, \Pi_\mathcal{T}) = \Omega(nx).\]
\end{proposition}
\begin{proof}
Let $\mathcal{A}$ be the set of points on the plane that covers $P$ and $\mathcal{L}$.
Using $P$ as a center, project $(\mathcal{X} \setminus \mathcal{A}) \cup \mathcal{L}$ onto a plane $M$ in general position.
Let $\mathcal{L'}$ be the projection of $\mathcal{L}$ on $M$.
Let $\mathcal{X'}$ be the projection of $\mathcal{X} \setminus \mathcal{A}$ on $M$.
Each point of $\mathcal{L}$ will project to a distinct point on $M$, so 
\[|\mathcal{L'}|=|\mathcal{L}| = |\mathcal{H}| - x' \geq \beta_3n - \beta_3x/2 > \beta_3n/2.\]

Since $P \in \mathcal{R}$, Proposition \ref{RExists} implies that $P$ is incident to $\Omega(x)$ lines spanned by $\mathcal{X} \setminus \mathcal{A}$.
Each of these lines will project to a distinct point on $M$, so
\[|\mathcal{X'}| = \Omega(x).\]

Let $l_m$ be the number of lines spanned on $M$ by $\mathcal{L'} \cup \mathcal{X'}$.
Each of these lines corresponds to a plane through $P$, so
\[\mathbf{I}(P,\Pi_\mathcal{T}) \geq l_m.\]

If $|\mathcal{X'}| > \beta_3x/2$, throw away all but $\beta_3x/2$ of the points of $\mathcal{X'}$ to ensure that the line containing $\mathcal{L'}$ has more points of $\mathcal{L'} \cup \mathcal{X'}$ than any other.
Lemma \ref{BeckErdos} implies that 
\[l_m = \Omega((|\mathcal{L'}|+|\mathcal{X'}|)|\mathcal{X'}|) = \Omega((\beta_3n/2)\Omega(x)) = \Omega(nx).\]
This completes the proof of Proposition \ref{IP}.
\end{proof}

Repeat the count for each of the $\Omega(x)$ points in $\mathcal{R}$ to get a lower bound on $\mathbf{I}(\mathcal{R}, \Pi_\mathcal{T})$, the number of incidences between the points of $\mathcal{R}$ and planes spanned by $\mathcal{T}$.
\begin{eqnarray*}
\mathbf{I}(\mathcal{R},\Pi_\mathcal{T}) & = & \sum_{P \in \mathcal{R}}\mathbf{I}(P,\Pi_\mathcal{T}) \\
& = & \sum_{i=0}^{|\mathcal{R}|}\Omega(nx) \\
& = & \Omega(nx^2)
\end{eqnarray*}
By Lemma \ref{bichromaticConsequence}, this lower bound implies that $\mathcal{T}$ determines $\Omega(nx^2)$ planes.

Since we have shown that the conclusion of the theorem holds when some line is incident to all but $c_mx = \min(\beta_3x/2, c_ex/2)$ points of $\mathcal{H}$, we may assume that no line is incident to more than $|\mathcal{H}|- c_mx$ points of $\mathcal{H}$ for the remaining cases.

\it Case 3. \rm
Suppose that the points of $\mathcal{H}$ lie in a plane.

Let $l_h$ be the number of lines spanned by $\mathcal{H}$.
Since $|\mathcal{H}| \geq \beta_3n$ and no line is incident to more than $|\mathcal{H}|- c_m x$ points of $\mathcal{H}$, Lemma \ref{BeckErdos} implies that
\[l_h \geq c_e|\mathcal{H}| c_m x \geq c_e c_m \beta_3 nx.\]
Each of the $x$ points in $\mathcal{X}$ is incident to a plane for each line spanned by $\mathcal{H}$, so
\[\mathbf{I}(\mathcal{X},\Pi_\mathcal{T}) \geq c_e c_m \beta_3 nx^2.\]
By Lemma \ref{bichromaticConsequence}, this implies that there are $\Omega (nx^2)$ planes spanned by the point set, proving the theorem for this case.

\it Case 4. \rm
If none of the previous cases hold, then the points of $\mathcal{H}$ must lie on two skew lines.
Let the sets of points on these lines be $\mathcal{L}_0$ and $\mathcal{L}_1$, with 
\[|\mathcal{L}_0| \geq |\mathcal{L}_1| \geq c_m x.\]
(If $\mathcal{L}_1 < c_m x$, we'd be in case 2.)
 Since $|\mathcal{H}| \geq \beta_3n$,
\[|\mathcal{L}_0| = |\mathcal{H}| - |\mathcal{L}_1| > \beta_3n/2.\]

Let P be a point in $\mathcal{X}$.
Using $P$ as a center, project the points of $\mathcal{H}$ onto a plane $M$ in general position.
Let $\mathcal{H'}$ be the projection of $\mathcal{H}$ on $M$.
Let $l_{h'}$ be the number of lines spanned by $\mathcal{H'}$.
Since $\mathcal{L}_0$ and $\mathcal{L}_1$ are skew and $P$ is on neither line, the points of $\mathcal{H'}$ will lie on crossing lines on the plane $M$.
The projected lines have at most $1$ point in common, so
\[l_{h'} \geq (|\mathcal{L}_0|-1)(|\mathcal{L}_1|-1) \geq \frac12|\mathcal{L}_0||\mathcal{L}_1| \geq \beta_3 c_m nx/4.\]
Each line spanned by the points of $\mathcal{H'}$ corresponds to a plane through $P$, so 
\[\mathbf{I}(P,\Pi_\mathcal{T}) \geq \beta_3 c_m nx/4.\]
The same logic holds for each of the $x$ points in $X$, so
\[\mathbf{I}(\mathcal{X},\Pi_\mathcal{T}) \geq \beta_3 c_m nx^2/4.\]
By Lemma \ref{bichromaticConsequence}, this implies that there are $\Omega (nx^2)$ hyperplanes spanned by the full point set.
\end{proof}

\section{Counterexample to Purdy's conjecture}
\label{sec:WeakPurdy}

In this section, we present an infinite family of counterexamples in dimensions 4 and higher to a conjecture of Purdy's, and we replace it with a modified version that is not subject to our counterexample.

Define the rank of a flat to be one more than its dimension. Purdy's conjecture \cite{Pur81, Pur86, EP95, BrMoPa05} was that, if $n$ is sufficiently large, then any set of $n$ points in $\mathbb{E}^d$, that cannot be covered by a set of flats whose ranks sum to $d+1$, spans at least as many hyperplanes as $(d-2)$-flats.
Gr\"unbaum found counterexamples for $d=3$ with up to 16 points \cite{GrSh84}, but until now no infinite family of counterexamples has been known.
\begin{theorem}\label{OldPurdy}
In $\mathbb{E}^d$, with $d \geq 4$, it is possible to construct a set of $n$ points, that cannot be covered by any set of flats whose ranks sum to $d+1$, that spans $\Theta_d(n^{d-2})$ hyperplanes and $\Theta_d(n^{d-1})$ $(d-2)$-flats.
\end{theorem}
\begin{proof}
We will construct a set of $n = k(d-1)$ points.
First, put $d-1$ lines in general position - we will refer to these lines as the covering lines.
On each of the covering lines, put $k$ points, so that any $d-1$ of these points, one on each of the covering lines, spans a $(d-2)$-flat.

The total rank of the covering lines is $2(d-1)$, which is greater than $d+1$ when $d > 3$.
Since the covering lines are in general position, no more than $m$ lines may be covered by a flat of rank $2m < d+1$, so there is no way to construct a set of flats covering the point set with total rank less than $d+1$.

We will count the number of hyperplanes spanned by the $n$ points by counting the number of hyperplanes that contain $j$ covering lines and summing over all possible values for $j$.

Let $h_j$ be the number of hyperplanes that contain $j$ covering lines.
A hyperplane that contains $j$ lines will contain $d-2j$ points from the remaining covering lines.
$h_j$ is equal to the number of ways to choose the $j$ lines contained in the hyperplane, times the number of ways to choose the $d-2j$ lines that each contribute $1$ point, times the number of ways to choose $1$ point on each of $d-2j$ lines.
In other words,
\[h_j = {d-1 \choose j}{d-1-j \choose d-2j}k^{d-2j}.\]

Let $h$ be the total number of hyperplanes spanned by the point set.
Since there are only $d-1$ covering lines, a hyperplane must contain at least one of the lines.
Since the lines are in general position, no hyperplane will contain more than $\lfloor d/2 \rfloor$ covering lines.
Consequently,
\begin{align*}
h &= \sum_{j} h_j \\
&= \sum_{1 \leq j \leq \lfloor d/2 \rfloor}{d-1 \choose j}{d-1-j \choose d-2j}k^{d-2j}.\\
\text{Since $k = \Theta_d(n)$,} \\
&= \sum_{1 \leq j \leq \lfloor d/2 \rfloor}\Theta_{d,j}(n^{d-2j}) \\
&= \Theta_d(n^{d-2}).\\
\end{align*}

We will proceed in a similar manner to count the number of $(d-2)$-flats spanned by the point set.

Let $g_j$ be the number of $(d-2($-flats that contain $j$ covering lines.
A $(d-2)$-flat that contains $j$ lines will contain $d-1-2j$ points from the remaining lines.
$g_j$ is equal to the number of ways to choose $j$ lines contained in the $(d-2)$-flat, times the number of ways to choose $d - 1 - 2j$ lines that each contribute $1$ point, times the number of ways to choose $1$ point on each of the $d-1-2j$ lines.
In other words,
\[g_j = {d-1 \choose j}{d - 1 - j \choose d - 1 -2j}k^{d-1-2j}.\]

Let $g$ be the total number of $(d-2)$-flats spanned by the point set.
Since there are $d-1$ covering lines, a $(d-2)$-flat may be spanned by one point from each line.
Since the lines are in general position, no $(d-2)$-flat will contain more than $\lfloor (d-1)/2 \rfloor$ covering lines.
Consequently,
\begin{align*}
g &= \sum_j g_j \\
&= \sum_{0 \leq j \leq \lfloor (d-1)/2 \rfloor}{d-1 \choose j}{d-1-j \choose d - 1 - 2j}k^{d-1-2j} \\
&= \sum_{0 \leq j \leq \lfloor (d-1)/2 \rfloor}\Theta_{d,j}(n^{d-1-2j}) \\
&= \Theta_d(n^{d-1}).\\
\end{align*}

\end{proof}

In place of the disproved conjecture, we propose the following related conjectures.

We say that a set of points is $r$-degenerate if it can be covered by any set of flats (of nonzero dimension) whose dimensions add up to less than $r$.
\begin{conjecture}[Modified weak Purdy]\label{higherWeakDirac}
A set of $n$ points in $\mathbb{E}^d$, that is not $d$-degenerate, contains a point incident to $\Omega(p)$ hyperplanes, where $p$ is the number of $(d-2)$-flats spanned by the set.
\end{conjecture}

Using a similar hypothesis, we conjecture a generalization of the Beck-Erd\H{o}s theorem that embraces our theorem in three dimensions as a special case.

\begin{conjecture}\label{higherBeckErdos}
A set of $n$ points in $\mathbb{E}^d$, having no $d$-degenerate subset of more than $n-k$ points, spans $\Omega(nk^{d-1})$ hyperplanes.
\end{conjecture}

\section{Acknowledgments}
We would like to thank Ryan Anderson for several helpful comments and suggestions.
We would also like to thank an anonymous referee for pointing out an error in the  proof of what is now Theorem \ref{BichromaticAgarwalAronov}, and for significant help with the presentation.

\bibliographystyle{IEEEtran}
\bibliography{HyperplanesSpannedByPointArrangements}

\begin{thebibliography}{10}
\providecommand{\url}[1]{#1}
\csname url@samestyle\endcsname
\providecommand{\newblock}{\relax}
\providecommand{\bibinfo}[2]{#2}
\providecommand{\BIBentrySTDinterwordspacing}{\spaceskip=0pt\relax}
\providecommand{\BIBentryALTinterwordstretchfactor}{4}
\providecommand{\BIBentryALTinterwordspacing}{\spaceskip=\fontdimen2\font plus
\BIBentryALTinterwordstretchfactor\fontdimen3\font minus
  \fontdimen4\font\relax}
\providecommand{\BIBforeignlanguage}[2]{{%
\expandafter\ifx\csname l@#1\endcsname\relax
\typeout{** WARNING: IEEEtran.bst: No hyphenation pattern has been}%
\typeout{** loaded for the language `#1'. Using the pattern for}%
\typeout{** the default language instead.}%
\else
\language=\csname l@#1\endcsname
\fi
#2}}
\providecommand{\BIBdecl}{\relax}
\BIBdecl

\bibitem{AA92}
P.~K. Agarwal and B.~Aronov, ``Counting facets and incidences,'' \emph{Discrete
  Comput. Geom.}, vol.~7, no.~4, pp. 359--369, 1992.

\bibitem{EP95}
P.~Erd\H{o}s and G.~Purdy, ``Extremal problems in combinatorial geometry,'' in
  \emph{Handbook of Combinatorics}, R.~Graham, M.~Gr\"{o}tschel, and
  L.~Lov\'{a}sz, Eds.\hskip 1em plus 0.5em minus 0.4em\relax Amsterdam:
  Elsevier, 1995, vol.~1, pp. 809--873.

\bibitem{SzTr83}
E.~Szemer\'edi and J.~W.~T. Trotter, ``Extremal problems in discrete
  geometry,'' \emph{Combinatorica}, vol.~3, no. 3-4, pp. 381--392, 1983.

\bibitem{PaSh04}
J.~Pach and M.~Sharir, ``Geometric incidences,'' in \emph{Towards a {T}heory of
  {G}eometric {G}raphs}, ser. Contemporary Mathematics, J.~Pach, Ed.\hskip 1em
  plus 0.5em minus 0.4em\relax American Mathematical Society, 2004, pp.
  185--223.

\bibitem{EdGuSh90}
H.~Edelsbrunner, L.~Guibas, and M.~Sharir, ``The complexity of many cells in
  arrangements of planes and related problems,'' \emph{Discrete and
  Computational Geometry}, vol.~5, no.~1, pp. 197--216, 1990.

\bibitem{EdSh90}
H.~Edelsbrunner and M.~Sharir, ``A hyperplane incidence problem with
  applications to counting distances,'' in \emph{SIGAL International Symposium
  on Algorithms}, 1990, pp. 419--428.

\bibitem{BrKn03}
P.~Brass and C.~Knauer, ``On counting point-hyperplane incidences,''
  \emph{Comput. Geom. Theory Appl.}, vol.~25, no. 1-2, pp. 13--20, 2003.

\bibitem{ElTo05}
G.~Elekes and C.~D. T\'{o}th, ``Incidences of not-too-degenerate hyperplanes,''
  in \emph{SCG '05: Proceedings of the twenty-first annual symposium on
  Computational geometry}.\hskip 1em plus 0.5em minus 0.4em\relax New York, NY,
  USA: ACM, 2005, pp. 16--21.

\bibitem{Ed87}
H.~Edelsbrunner, \emph{Algorithms in Combinatorial Geometry}.\hskip 1em plus
  0.5em minus 0.4em\relax Springer-Verlag, 1987.

\bibitem{Be83}
J.~Beck, ``On the lattice property of the plane and some problems of {D}irac,
  {M}otzkin and {E}rd{\H{o}}s in combinatorial geometry,''
  \emph{Combinatorica}, vol.~3, no.~3, pp. 281--297, 1983.

\bibitem{Han65}
S.~Hansen, ``A generalization of a theorem of {S}ylvester on the lines
  determined by a finite point set,'' \emph{Mathematica Scandinavica}, vol.~16,
  pp. 175--180, 1965.

\bibitem{Dir51}
G.~A. Dirac, ``Collinearity properties of sets of points,'' \emph{Quarterly
  Journal of Mathematics}, vol.~2, no.~1, pp. 221–--227, 1951.

\bibitem{Mot51}
T.~Motzkin, ``The lines and planes connecting the points of a finite set,''
  \emph{Transactions of the American Mathematical Society}, vol.~70, no.~3, pp.
  451--464, 1951.

\bibitem{Er75}
P.~Erd\H{o}s, ``On some problems of elementary and combinatorial geometry,''
  \emph{Annali di Matematica pur ed applicata Ser. IV}, vol. CIII, pp. 99--108.

\bibitem{Er81}
------, ``On the combinatorial problems which {I} would most like to see
  solved,'' \emph{Combinatorica}, vol.~1, pp. 24--42, 1981.

\bibitem{Pur81}
G.~Purdy, ``A proof of a consequence of {D}irac's conjecture,''
  \emph{Geometriae Dedicata}, vol.~10, pp. 317--321, 1981.

\bibitem{Pur86}
------, ``Two results about points, lines and planes,'' \emph{Discrete Math.},
  vol.~60, pp. 215--218, 1986.

\bibitem{BrMoPa05}
P.~Brass, W.~Moser, and J.~Pach, \emph{Research {P}roblems in {D}iscrete
  {G}eometry}.\hskip 1em plus 0.5em minus 0.4em\relax Springer-Verlag, 2005.

\bibitem{GrSh84}
B.~Gr{\"u}nbaum and G.~Shephard, ``{Simplicial arrangements in projective
  3-space},'' \emph{Mitt. Math. Sem. dessen}, vol. 166, pp. 49--101, 1984.

\end{thebibliography}

\end{document}